\documentclass[12 pt]{amsart}

\usepackage{graphicx}
\usepackage{color}
\usepackage{amssymb, amsmath, amsthm}
\usepackage{mathptmx}

\usepackage{epsfig}
\usepackage{epstopdf}
\usepackage{graphicx}

\theoremstyle{plain}
\newtheorem{definition}{Definition}
\newtheorem{example}{Example}
\newtheorem{theorem}{Theorem}

\newtheorem{proposition}{Proposition}

\newtheorem{conjecture}{Conjecture}

\newcommand{\R}{\mathbb R} 

\newcommand{\bi}{\begin{itemize}}
\newcommand{\ei}{\end{itemize}}
\newcommand{\be}{\begin{enumerate}}
\newcommand{\ee}{\end{enumerate}}

\newcommand{\pd}{\partial}

\newcommand{\K}{\mathcal{K}}

\begin{document}
\title{Natural properties of the trunk of a knot}
\author{Derek Davies}
\author{Alexander Zupan}

\maketitle

\begin{abstract}
The trunk of a knot in $S^3$, defined by Makoto Ozawa, is a measure of geometric complexity similar to the bridge number or width of a knot.  We prove that for any two knots $K_1$ and $K_2$, we have $tr(K_1 \# K_2) = \max\{tr(K_1),tr(K_2)\}$, confirming a conjecture of Ozawa.   Another conjecture of Ozawa asserts that any width-minimizing embedding of a knot $K$ also minimizes the trunk of $K$.  We produce several families of probable counterexamples to this conjecture.
\end{abstract}

\section{Introduction}

The \emph{width} $w(K)$ of a knot $K$ in the 3-sphere was introduced in \cite{gabai} as crucial tool in Gabai's proof of the Property R Conjecture.  Since its inception, width has spawned an entire theory in 3-manifold topology, known as \emph{thin position} theory.  The bridge number of a knot, a precursor to width introduced by Schubert, is another well-studied and well-understood measure of the geometric complexity of a knot.  In this paper, we examine a newer invariant, Ozawa's \emph{trunk} of a knot \cite{ozawa}.  Trunk, like width and bridge number, is computed by minimizing some complexity over all possible embeddings of a knot.

One problem in this setting is to determine the behavior of a given invariant under standard operations, such as taking satellites or connected sums.  The effects of taking satellites on bridge number and width have been studied extensively; see for instance \cite{liguo,schubert,schultens,zupan2,zupan1}.  In the more specific case of connected summation, a natural question arises: if $K = K_1 \# K_2$, does the obvious embedding of $K$ obtained by stacking minimal embeddings of $K_1$ and $K_2$ minimize the complexity of $K$?  For the bridge number, $b(K)$, the answer is ``yes": following \cite{schubert, schultens}, we have $b(K_1 \# K_2) = b(K_1) + b(K_2) - 1$.  For width, the answer is ``often but not always": for many knots, $w(K_1 \# K_2) = w(K_1) + w(K_2) - 2$ \cite{riecksedg}, but there exist knots $K_1$ and $K_2$ for which $w(K_1 \# K_2)  = \max\{ w(K_1),w(K_2)\} < w(K_1) + w(K_2) - 2$ \cite{blairtom}.

In the present work, we prove that the trunk of a knot behaves as expected under taking connected sums, confirming a conjecture made by Ozawa in \cite{ozawa}.

\begin{theorem}\label{thm:trunk}
For any two knots $K_1$ and $K_2$ in $S^3$,
\[ tr(K_1 \# K_2) = \max\{ tr(K_1),tr(K_2)\}.\]
\end{theorem}

The proof relies on a reimbedding theorem of Scharlemann and Schultens initially proved to study knot width \cite{scharschult}.

We also examine another conjecture made by Ozawa in \cite{ozawa}, which asserts that for every knot $K$, an embedding with minimal width also has minimal trunk.  We produce potential counterexamples to this conjecture, stopping short of proving that these embeddings minimize width.  We give concrete examples which are likely to satisfy the following conjecture:

\begin{conjecture}
There exists a knot $K$ with an embedding $k$ such that $k$ minimizes width but does not minimize trunk or bridge number, and there exists a knot $K'$ with an embedding $k'$ such that $k'$ minimizes width and bridge number but does not minimize trunk.
\end{conjecture}

\subsection{Acknowledgements}
The second author is partially supported by NSF grant DMS--1203988.

\section{Preliminaries}

We begin this section by defining the three related knot invariants discussed in the introduction, \emph{bridge number}~\cite{schubert}, \emph{width}~\cite{gabai}, and \emph{trunk}~\cite{ozawa}.  Each invariant minimizes a different measure of geometric complexity over the set of all embeddings isotopic to a particular knot.  For the remainder of the paper, we will fix a Morse function $h:S^3 \rightarrow \R$ such that $h$ has exactly two critical points, which are denoted $\pm \infty$.

\begin{definition}
Fix a knot $K$ and let $\K$ denote the set of all embeddings $k$ of $S^1$ into $S^3$ isotopic to $K$ such that $h_k$ is Morse.  For each $k \in \K$, let $c_0 < \dots < c_p$ denote the critical values of $h_k$, and choose regular values $c_0 < r_1 < c_1 < \dots < c_{p-1} < r_p < c_p$.  For each $i$, $1 \leq i \leq p$, let $x_i = |h^{-1}(r_i) \cap k|$, the width of the level set $h^{-1}(r_i)$.

The bridge number $b(k)$ of $k$ is the number of maxima (or minima) of $h_k$, the width $w(k)$ of $k$ is $w(k) = \sum x_i$, and the trunk $tr(k)$ of $k$ is $tr(k) = \max \{x_i\}$.  The invariants bridge number $b(K)$, width $w(K)$, and trunk $tr(K)$ are defined as follows:
\begin{eqnarray*}
b(K) &=& \min_{k \in \K} b(k) \\
w(K) &=& \min_{k \in \K} w(k) \\
tr(K) &=& \min_{k \in \K} tr(k).
\end{eqnarray*}
For a particular embedding $k \in \K$, we say that $h^{-1}(r_i)$ is a thick level if $x_i > x_{i-1},x_{i+1}$, and $h^{-1}(r_i)$ is a thin level if $x_i < x_{i-1},x_{i+1}$.
\end{definition}

If $k$ is an embedding of $K$ that realizing minimal width (that is, $w(k) = w(K)$), we call $k$ a \emph{thin position} of $K$.  

Next, we define the satellite construction, of which the connected sum operation is a special case.

\begin{definition}
Let $J$ be a knot in $S^3$, suppose that $V$ is an unknotted solid torus with core $C$ containing a knot $\hat{K}$ that meets every meridian disk of $V$, and let $\varphi:V \rightarrow S^3$ be a knotted embedding, where the image $\varphi(C)$ of $C$ is isotopic to $J$ in $S^3$.  Then $K = \varphi(\hat{K})$ is called a satellite knot with companion $J$ and pattern $\hat{K}$.
\end{definition}

Note that if $K = K_1 \# K_2$ for knots $K_1,K_2$ in $S^3$, then $K$ is a satellite knot with companion $K_i$ and pattern $K_j$, where $\{i,j\} = \{1,2\}$ and $K_j$ has been embedded in a solid torus so that a meridian disk meets $K_j$ in a single point.

We need one final cluster of definitions in order to state known results.

\begin{definition}
Let $K$ be a knot in $S^3$, with $N(K)$ an open regular neighborhood of $K$, and let $E(K) = S^3 \setminus N(K)$ be the exterior of $K$ in $S^3$.  A properly embedded, orientable surface $S \subset E(K)$ is meridional if the curves $\pd S \subset \overline{N(K)}$ bound meridian disks of the solid torus $\overline{N(K)}$, and $S$ is incompressible if the induced inclusion map $i_*:\pi_1(S) \rightarrow \pi_1(E(K))$ is injective.  The knot $K$ is called meridionally-planar small (or mp-small) if $E(K)$ does not contain an incompressible, planar, meridional surface $S$.
\end{definition}

\section{Behavior under connected sums}

In order to understand the behavior of bridge number, width, and trunk under the connected sum operation, we first describe a straightforward upper bound.

\begin{definition}
Let $K_1$ and $K_2$ be knots in $S^3$ with embeddings $k_1$ and $k_2$, respectively.  Let $M$ be the point on $K_1$ corresponding to the maximum of $h_{k_1}$, and let $m$ be the point on $k_2$ corresponding to the minimum of $h_{k_2}$.  We obtain an embedding $k$ of $K = K_1 \# K_2$ via the following process:  Reimbed $k_1$ and $k_2$ into $S^3$ so that $h(M) < h(m)$, and connect $M$ to $m$ with two vertical arcs, so that the resulting embedding $k$ is isotopic to $K$.  We say that $k$ is obtained by stacking $k_2$ on $k_1$.  See Figure \ref{fig:ex1}.
\end{definition}
By stacking minimal bridge positions, thin positions, or minimal trunk positions of two knots $K_1$ and $K_2$, we immediately obtain the following inequalities:
\begin{eqnarray*}
b(K_1 \# K_2) &\leq& b(K_1) + b(K_2) - 1\\
w(K_1 \# K_2) &\leq& w(K_1) + w(K_2) - 2 \\
tr(K_1 \# K_2) &\leq& \max\{tr(K_1),tr(K_2)\}.
\end{eqnarray*}

\begin{example}
Consider the alternating knots $K_1 = 4_1$ and $K_2 = 8_5$.  We have $b(K_1) = 2$, $w(K_1) = 8$, and $tr(K_1) = 4$, while $b(K_2) = 3$, $w(K_2) = 18$, and $tr(K_2) = 6$.  The stacked embedding shown in Figure \ref{fig:ex1} minimizes bridge number, width, and trunk for $K = K_1 \# K_2$, where $b(K) = 4$, $w(K) = 24$, and $tr(K) = 6$ (in this case, each inequality above is sharp).
\begin{figure}[h!]
  \centering
    \includegraphics[width=.3\textwidth]{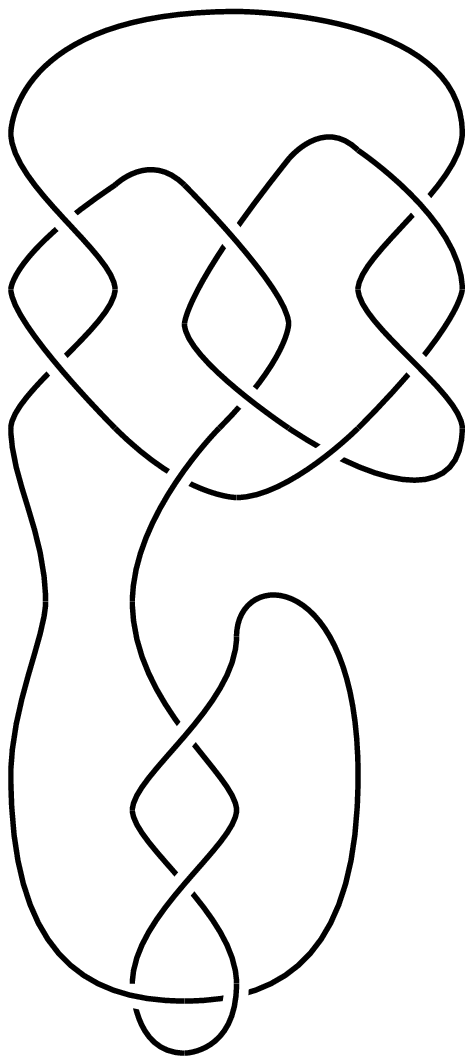} 
\caption{Stacked embeddings of $K_1 = 4_1$ and $K_2=  8_5$}
    \label{fig:ex1}
\end{figure}
\end{example}

Naturally, one might wonder whether stacking is the best we can do for each of the three invariants.  For bridge number, this is indeed optimal, as shown by Schubert (with an updated proof given by Schultens):

\begin{theorem}\cite{schubert,schultens}
For any two knots $K_1$ and $K_2$ in $S^3$,
\[ b(K_1 \# K_2) = b(K_1) + b(K_2) - 1.\]
\end{theorem}

For width, the picture is somewhat murkier.  First, Rieck and Sedgwick proved the following:

\begin{theorem}\cite{riecksedg}
If $K_1$ and $K_2$ are mp-small knots in $S^3$, then
\[ w(K_1 \# K_2) = w(K_1) + w(K_2) - 2.\]
\end{theorem}

On the other hand, using examples exhibited by Scharlemann and Thompson~\cite{scharthomp}, Blair and Tomova proved that stacking does not always give minimal width.

\begin{theorem}\cite{blairtom}
There exist knots $K_1$ and $K_2$ in $S^3$ with the property that
\[ w(K_1 \# K_2) = \max\{w(K_1),w(K_2)\} < w(K_1) + w(K_2) - 2.\]
\end{theorem}

In~\cite{ozawa}, Ozawa conjectured that the upper bound obtained by stacking is optimal with respect to trunk; namely

\begin{conjecture}\cite{ozawa}
For any two knots $K_1$ and $K_2$ in $S^3$,
\[ tr(K_1 \# K_2) = \max\{tr(K_1),tr(K_2)\}.\]
\end{conjecture}

Following the work of Rieck and Sedgwick referenced above~\cite{riecksedg}, Ozawa proved that the conjecture is true for mp-small knots.

In the present work, we prove Ozawa's conjecture.  To do this, we require machinery originally developed by Scharlemann and Schultens to understand the behavior of width under connected sums~\cite{scharschult}.  Although their theorem can be applied more generally, we state it using the formulation in~\cite{hendricks}.

\begin{theorem}\cite{hendricks, scharschult}\label{reimbed}
Let $h$ be the standard Morse function on $S^3$, and let $\hat{K}$ be a knot in an unknotted solid torus $V \subset S^3$.  For every possibly knotted embedding $\varphi: V \rightarrow S^3$, there is a reimbedding $\varphi': V \rightarrow S^3$ such that
{\be
\item $h \circ \varphi = h \circ \varphi'$ on the solid torus $V$, 
\item $\varphi'(V)$ is an unknotted solid torus, and 
\item $\varphi'(\hat{K})$ is isotopic to $\hat{K}$ in $S^3$.
\ee}
\end{theorem}
The first conclusion, $h \circ \varphi = h \circ \varphi'$ on $V$, implies that the reimbedding $\varphi'$ is \emph{height-preserving}.  More specifically, let $k = \varphi(\hat{K})$ and $k' = \varphi'(\hat{K})$.  Then $k'$ is isotopic to $\hat{K}$, and $k'$ is a satellite knot with pattern $\hat{K}$.  Moreover, $h_k$ and $h_{k'}$ have the same critical values $c_0 < \dots < c_n$, and for regular values $c_0 < r_1 < \dots < r_n < c_n$, we have $|h^{-1}(r_i) \cap k| = |h^{-1}(r_i) \cap k'|$ for all $i$, so that the combinatorial data carried by $h_k$ is identical to that of $h_{k'}$ -- in particular, $tr(k) = tr(k')$.  We are now equipped to prove Ozawa's conjecture.

\begin{proof}[Proof of Theorem \ref{thm:trunk}]
As noted above, $tr(K_1 \# K_2) \leq \max\{tr(K_1),tr(K_2)\}$; hence we need to show that we cannot do any better than this bound.  For this purpose, let $k$ be an embedding of $K_1 \# K_2$ such that $tr(k) = tr(K_1 \# K_2)$.  Since $K_1 \# K_2$ may be viewed as a satellite knot, where the pattern $\hat{K}$ is $K_1$ and the companion is $K_2$, there is a solid torus $V$ containing $\hat{K} = K_1$ and a knotted embedding $\varphi: V \rightarrow S^3$ such that $\varphi(\hat{K}) = k$.

By Theorem~\ref{reimbed}, there is a reimbedding $\varphi':V \rightarrow S^3$ so that $h \circ \varphi = h \circ \varphi'$ on $V$ and $k' = \varphi(\hat{K})$ is isotopic to $\hat{K} = K_1$.  By the discussion following Theorem~\ref{reimbed}, we have $tr(k) = tr(k')$, and thus
\[ tr(K_1 \# K_2) = tr(k) = tr(k') \geq tr(K_1).\]
A parallel argument replaces $K_1$ with $K_2$, and thus
\[ tr(K_1 \# K_2) \geq \max\{tr(K_1),tr(K_2)\}.\]
Taken together, the two inequalities yield
\[ tr(K_1 \# K_2) = \max\{tr(K_1),tr(K_2)\},\]
as desired.
\end{proof}

\section{Thin position versus trunk position}

In this section, we examine another of Ozawa's conjectures:

\begin{conjecture}\cite{ozawa}\label{thintrunk}
Suppose $k$ is a thin position of $K$, so that $w(K) = w(k)$.  Then $k$ is also a minimal trunk position; that is, $tr(K) = tr(k)$.
\end{conjecture}

We will produce several potential counterexamples to Ozawa's conjecture.  Unfortunately, the width of these examples is prohibitively large (greater than 200); hence, we make no attempt to prove that the conjectured embeddings are, in fact, thin position.

Consider the templates pictured below in Figures \ref{fig:embed1} and \ref{fig:embed2}, which were introduced in \cite{scharthomp} and further examined in \cite{blairtom2} and \cite{blairtom}.  Each embedding is equipped with a natural height function $h$, projection onto a vertical axis, and each box $B_i$ represents a \emph{braid}, a collection of arcs containing no critical points and connecting the top and bottom strands of the box.  As pictured, the parameters $r_1$ and $r_2$ represent some numbers of critical points and $s_1$ and $s_2$ represent some numbers of parallel strands.  For fixed braids, the embeddings $k_{r_1,r_2,s_1,s_2}$ and $k'_{r_1,r_2,s_1,s_2}$ are isotopic; we let $K_{r_1,r_2,s_1,s_2}$ denote the knot corresponding to these embeddings.

\begin{figure}[h!]
  \centering
    \includegraphics[width=.55\textwidth]{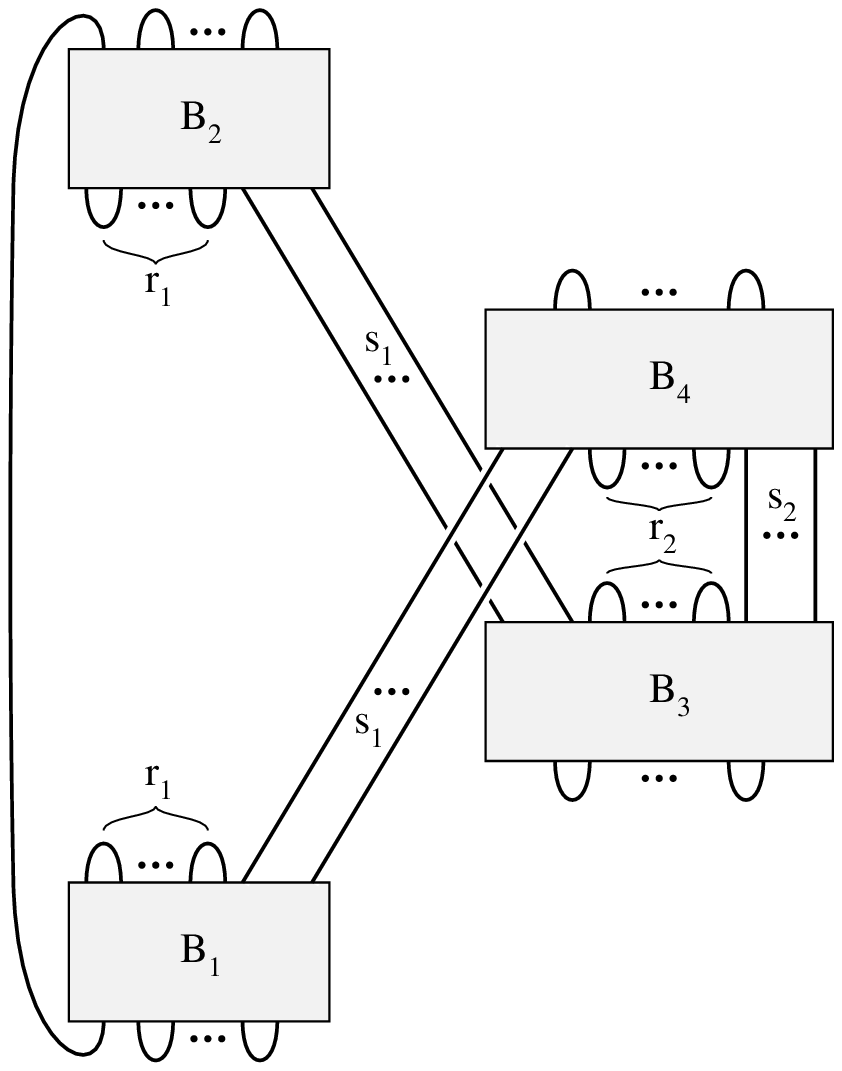} 
    \caption{The embedding $k_{r_1,r_2,s_1,s_2}$.}
    \label{fig:embed1}
\end{figure}

\begin{figure}[h!]
  \centering
 \includegraphics[width=.55\textwidth]{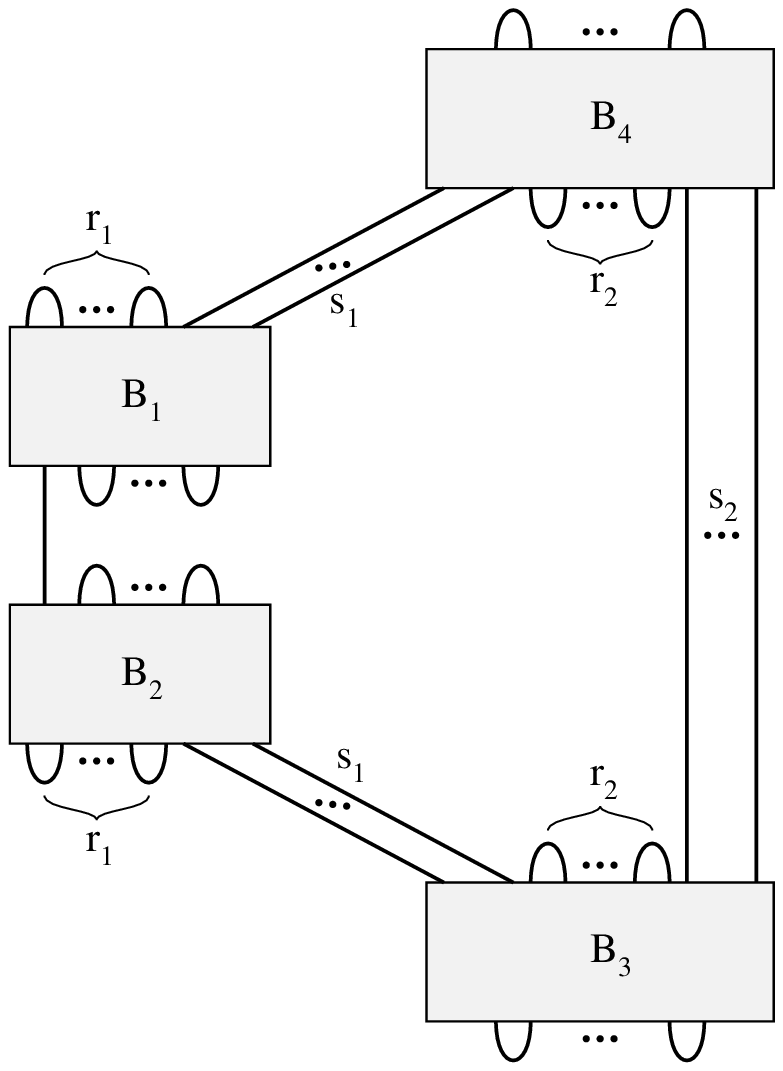} 
    \caption{The embedding $k'_{r_1,r_2,s_1,s_2}$.}
    \label{fig:embed2}
\end{figure}

In \cite{blairtom}, it was shown that
\begin{theorem}\cite{blairtom}
For sufficiently complicated braids $B_i$, the $w(K_{3,0,3,3}) = 134$ and $k_{3,0,3,3}$ is a thin position of $K_{3,0,3,3}$.
\end{theorem}
We do not reproduce the technical conditions which give rise to a rigorous definition of a \emph{sufficiently complicated braid}; instead, we refer the interested reader to \cite{blairtom}.  While the primary purpose of this theorem was to demonstrate that width is not additive under taking connected sums, it also provided the first example of a knot $K_{3,0,3,3}$ with a thin position $k_{3,0,3,3}$ such that $b(k_{3,0,3,3}) > b(K_{3,0,3,3})$, showing that width and bridge number need not be realized simultaneously.  In a similar vein, we will give evidence that neither trunk and width nor trunk and bridge number need be realized simultaneously for every knot $K$.

\begin{proposition}\label{widthtrunk}
For parameters $r_1 = 2, r_2 = 1, s_1 = 3, s_2=7$, we have
\begin{eqnarray*}
tr(k_{2,1,3,7}) = 16 &>& tr(k'_{2,1,3,7}) = 14 \\
w(k_{2,1,3,7}) = 206 &<& w(k'_{2,1,3,7}) = 208 \\
b(k_{2,1,3,7}) = 13 &>& b(k'_{2,1,3,7}) = 12.
\end{eqnarray*}
For parameters $r_1 = 4, r_2 = 1, s_1 = s_2 = 3$, we have
\begin{eqnarray*}
tr(k_{4,1,3,3}) = 12 &<& tr(k'_{4,1,3,3}) = 14 \\
w(k_{4,1,3,3}) = 222 &>& w(k'_{4,1,3,3}) = 216 \\
b(k_{4,1,3,3}) = 15 &>& b(k'_{4,1,3,3}) = 14.
\end{eqnarray*}
\end{proposition}
\begin{proof}
As shown in \cite{scharschult}, the width of an embedding $k$ can be computed from the widths $\{a_i\}$ of its thick levels and the widths $\{b_j\}$ of its thin levels:
\[ w(k) = \frac{1}{2} \sum a_i^2 - \frac{1}{2} \sum b_j^2.\]
In addition, we use a well-known formula for bridge number:
\[ b(k) = \frac{1}{2} \sum a_i - \frac{1}{2} \sum b_j.\]
Observe that $k_{2,1,3,7}$ has thick levels of widths $\{8, 16, 16, 8\}$ and thin levels of widths $\{4, 14, 4\}$; hence $tr\left(k_{2,1,3,7}\right) = 16$, and we have
\begin{eqnarray*}
w(k_{2,1,3,7}) &=& \frac{1}{2} \left( 8^2 + 16^2 +16^2 +8^2 - 4^2 - 14^2 - 4^2 \right) = 206 \\
b(k_{2,1,3,7}) &=& \frac{1}{2} \left( 8 + 16 +16 +8 - 4 - 14 - 4 \right) = 13.
\end{eqnarray*}
On the other hand, $k'_{2,1,3,7}$ has thick levels of widths $\{12, 14, 14, 12\}$ and thin levels of widths $\{10, 8, 10\}$; hence $tr(k'_{2,1,3,7}) = 14$, and we have
\begin{eqnarray*}
w(k'_{2,1,3,7}) &=& \frac{1}{2} \left( 12^2 + 14^2 +14^2 +12^2 - 10^2 - 8^2 - 10^2 \right) = 208 \\
b(k'_{2,1,3,7}) &=& \frac{1}{2} \left( 12 + 14 +14 +12 - 10 - 8 - 10 \right) = 12.
\end{eqnarray*}
The embedding $k_{4,1,3,3}$ has thick/thin level widths $\{12, 12, 12, 12\}$ and $\{4, 10, 4\}$, while the embedding $k'_{4,1,3,3}$ has thick/thin level widths $\{8, 14, 14, 8\}$ and $\{6, 4, 6\}$.  The corresponding calculations are similar.
\end{proof}

In general, showing that a conjectured thin position is actually thin position is a difficult proposition, especially for knots with large width.  The theory in \cite{blairtom} develops techniques in this direction, but this work verifies thin position for a class of knots $K$ such that $w(K) = 134$.  In \cite{blairzupan}, the authors use many of these techniques to find thin position for another class of knots $K$ with $w(K) = 78$.  Verifying that the knots from Theorem \ref{widthtrunk} have widths 206 or 216 is -- in our opinion -- beyond the limits of the current technology.  However, when the braids $B_i$ are sufficiently complicated, it is highly likely that minimal trunk, width, and bridge number are realized by one of the two embeddings for which they are computed in Proposition \ref{widthtrunk}, giving rise to the following conjecture:

\begin{conjecture}
For sufficiently complicated braids $B_i$, thin position of the knot $K_{2,1,3,7}$ does not realize minimal trunk or minimal bridge number.  For sufficiently complicated braids $B_i$, thin position of the knot $K_{4,1,3,3}$ realizes minimal bridge number but not minimal trunk.
\end{conjecture}
If true, Conjecture 1 holds, where $K = K_{2,1,37}$ and $K' = K_{4,1,3,3}$.

\bibliographystyle{amsplain}
\bibliography{references}

\end{document}